\documentclass[review]{elsarticle}
\usepackage{amsmath}
\usepackage{graphicx,setspace}
\usepackage{amsfonts,amsmath, amssymb, amsthm}
\usepackage{mathrsfs}
\usepackage{epstopdf}
\usepackage{float}
\usepackage{lineno,hyperref}
\usepackage{color}
\usepackage{subfigure}
\usepackage{bm}
\usepackage{graphicx}
\usepackage{amssymb, amsthm}
\usepackage{mathrsfs}
\usepackage{epstopdf}
\usepackage{float}
\usepackage{lineno,hyperref}
\usepackage{color}
\usepackage{subfigure}
\usepackage{cases}
\usepackage{geometry}
\usepackage{pifont}
\numberwithin{equation}{section}

\journal{}

\begin{document}

\newtheorem{theorem}{Theorem}[section]
\newtheorem{assumption}{Assumption}
\newtheorem{definition}[theorem]{Definition} 
\newtheorem{lemma}[theorem]{Lemma} 
\newtheorem{corollary}[theorem]{Corollary}
\newtheorem{example}{Example}[section]
\newtheorem{proposition}[theorem]{Proposition}
\newtheorem{remark}{Remark}[section]

\def\e{\varepsilon}
\def\Rn{\mathbb{R}^{n}}
\def\Rm{\mathbb{R}^{m}}
\def\E{\mathbb{E}}
\def\hte{\bar\theta}
\def\cC{{\mathcal C}}
\newcommand{\Cp}{\mathbb C}
\newcommand{\R}{\mathbb R}
\renewcommand{\L}{\mathcal L}
\newcommand{\supp}{\mathrm{supp}\,}

\numberwithin{equation}{section}

\begin{frontmatter}

\title{{\bf Weak pullback attractors for damped stochastic fractional Schr\"odinger equation on $\mathbb{R}^n$
}}
\author{\normalsize{\bf Ao Zhang$^{a,}\footnote{ aozhang1993@csu.edu.cn}$,
Yanjie Zhang$^{b}\footnote{zhangyj2022@zzu.edu.cn}$,
Sanyang Zhai$^{b}\footnote{zhai20030126@gs.zzu.edu.cn}$,
Li Lin$^{c,}\footnote{linlimath@whut.edu.cn}$,
} \\[10pt]
\footnotesize{${}^a$ School of Mathematics and Statistics and HNP-LAMA,  Central South University,  Changsha 410083, China} \\[5pt]
\footnotesize{${}^b$ School of Mathematics and Statistics, Zhengzhou University 450001,  China.}  \\[5pt]
\footnotesize{${}^c$ College of Science, Wuhan University of Technology, Wuhan 430074, China.} \\[5pt]
}

\begin{abstract}
This article discusses the weak pullback attractors for a damped stochastic fractional Schrödinger equation on $\mathbb{R}^n$ with $n\geq 2$. By utilizing the stochastic Strichartz estimates and a stopping time technique argument, the existence and uniqueness of a global solution for the systems with the nonlinear term $|u|^{2\sigma}u$ are proven. Furthermore, we define a mean random dynamical system due to the uniqueness of the solution, which has a unique weak pullback mean random attractor in $L^\rho\left(\Omega; L^2\left(\mathbb{R}^n\right)\right)$. This result highlights the long-term dynamics of a broad class of stochastic fractional dispersion equations.

\end{abstract}


\end{frontmatter}

%\linenumbers

\section{Introduction.}

This article is devoted to studying the asymptotic dynamics for a dynamical system generated by a nonlinear fractional Schr\"odinger type equation with additive noise that reads
\begin{equation}{\label{li}}
\left\{
\begin{aligned}
&i du-\left[(-\Delta)^{\alpha} u-|u|^{2\sigma}u\right] dt+i\gamma udt=f(t,x,u)dt+ dW(t), \quad x \in \mathbb{R}^{n}, \quad t \geq 0, \\
&u(0)=u_0,
\end{aligned}
\right.
\end{equation}
where $u$ is a complex valued process defined on $\mathbb{R}^{n} \times \mathbb{R}^{+}$, $i=\sqrt{-1}$, $0<\sigma,\gamma <\infty$, $f$ is a given function on $\mathbb{R}\times \mathbb{R}^{n}\times \mathbb{R}$, and $W(t)$ is a $Q$-Wiener process. The fractional Laplacian operator $(-\Delta)^{\alpha}$ with admissible exponent $\alpha \in (0,1)$ is involved.

It would be helpful to provide a mathematical definition of $W(t)$. Let $(\Omega, \mathcal{F}, \mathbb{P})$ be a probability space, $\left(\mathcal{F}_{t}\right)_{t \geq 0}$ be a filtration, and let $\left(\beta_{k}\right)_{k \in \mathbb{N}}$ be a sequence of independent Brownian motions associated to this filtration. Given $\left(e_{k}\right)_{k \in \mathbb{N}}$ an orthonormal basis of $L^{2}\left(\mathbb{R}^{n}, \mathbb{R}\right)$, and a linear operator $\Phi$ on $L^{2}\left(\mathbb{R}^{n}, \mathbb{R}\right)$ with a real-valued kernel $k$:
\begin{equation}
\Phi h(x)=\int_{\mathbb{R}^{n}} k(x, y) h(y) d y, \quad h \in L^{2}\left(\mathbb{R}^{n}, \mathbb{R}\right).
\end{equation}
Then the process
\begin{equation*}
W(t, x, \omega):=\sum_{k=0}^{\infty} \beta_{k}(t, \omega) \Phi e_{k}(x), \quad t \geq 0, \quad x \in \mathbb{R}^{n}, \quad \omega \in \Omega
\end{equation*}
is a Wiener process on $L^{2}\left(\mathbb{R}^{n}, \mathbb{R}\right)$ with covariance operator $\Phi \Phi^{*}$.

The fractional nonlinear Schr\"odinger equation finds applications in various fields such as nonlinear optics \cite{Longhi15}, quantum physics \cite{BM23} and water propagation \cite{Pusateri14}. Inspired by the Feynman path approach to quantum mechanics, Laskin \cite{Laskin02} used the path integral over L\'evy-like quantum mechanical paths to obtain a fractional Schr\"odinger equation.  Kirkpatrick et al. \cite{Kirkpatrick02} considered a general class of discrete nonlinear Schr\"odinger equations on the lattice $h\mathbb{Z}$ with mesh size $h>0$, they showed that the limiting dynamics were given by a nonlinear fractional  Schr\"odinger equation when $h\rightarrow 0$. Guo and Huang \cite{Guo12} applied concentration compactness and commutator estimates to obtain the existence of standing waves for nonlinear fractional Schr\"odinger equations under some assumptions. Shang and Zhang \cite{SZ} proved that the critical fractional Schr\"odinger equation had a nonnegative ground state solution, and gave the relation between the number of solutions and the set's topology.  Choi and Aceves \cite{CA23} proved that the solutions of the discrete nonlinear Schr\"odinger equation with non-local algebraically decaying coupling converged strongly in $L^2(\mathbb{R}^n)$ to those of the continuum fractional nonlinear Schr\"odinger equation.  Frank and his collaborators \cite{LS15} proved general uniqueness results for radial solutions of linear and nonlinear equations involving the fractional Laplacian $(-\Delta)^{\alpha}$ with $\alpha \in(0,1)$ for any space dimensions $N\geq 1$. Wang and Huang \cite{Huang15} proposed an energy conservative difference scheme for the nonlinear fractional Schr\"odinger equations and gave a rigorous analysis of the conservation property. Chen and Guo \cite{CG21} studied a regularized Lie-Trotter splitting spectral method for a regularized space-fractional logarithmic Schr\"odinger equation by introducing a small regularized parameter. An and Yang \cite{AY23} showed a connection between fractional Schr\"odinger equations with power-law nonlinearity and fractional Schr\"odinger equations with logarithm-law nonlinearity.

There has been a lot of interest in the study of asymptotic behavior for fractional Schr\"odinger equation. B. Alouini \cite{BA21, BA22} focused on the asymptotic dynamics for a dynamical system generated by a nonlinear dispersive fractional Schr\"odinger type equation and proved the existence of regular finite dimensional global attractor in the phase space with finite fractal dimension and the energy space. Goubet and Zahrouni \cite{GZ17} considered a weakly damped forced nonlinear fractional Schr\"odinger equation and proved that this global attractor had a finite fractal dimension if the external force was in a suitable weighted space.  Gu et al. \cite{AG18} investigated the regularity of random attractors for the non-autonomous non-local fractional stochastic reaction-diffusion equations in $H^{s}(\mathbb{R}^n)$ with $s \in (0,1)$.

In the framework of stochastic mechanics, there are few results on the well-posedness and asymptotic behavior of damped stochastic fractional nonlinear Schr\"odinger equation in $L^2(\mathbb{R}^n)$,  due to the complexity brought by the fractional Laplacian operator $(-\Delta)^{\alpha}$, the nonlinear term $|u|^{2\sigma}u$ and white noise. Yuan and Chen \cite{Chen17} proved the existence of martingale solutions for the stochastic fractional nonlinear Schr\"odinger equation on a bounded interval.  As far as the author is aware, there is no result reported in the literature for the stochastic system \eqref{li} with a nonlinear diffusion term $|u|^{2\sigma}u$ and fractional Laplacian operator $(-\Delta)^{\alpha}$ on an unbounded interval. 
Here we are particularly interested in the long-term
dynamics of the stochastic system \eqref{li}, especially the existence and uniqueness of weak pullback attractor for the damped fractional
Schr\"odinger equation with additive noise. More precisely, we will prove the existence and uniqueness of weak pullback mean random attractor of the corresponding mean random dynamical system associated with \eqref{li} in $L^\rho\left(\Omega; L^2\left(\mathbb{R}^n\right)\right)$. To study the weak pullback attractor of \eqref{li}, we must establish the well-posedness of the stochastic system \eqref{li}. In contrast to the case of stochastic nonlinear Schr\"odinger equation in $L^2(\mathbb{R}^n)$.  There are some essential difficulties in our problems. The first difficulty is the appearance of the fractional Laplacian operator $(-\Delta)^{\alpha}$. The deterministic cubic fractional nonlinear Schr\"odinger equation is ill-posedness in $L^{2}(\mathbb{R}^{n})$ (see e.g., \cite{Cho15}). The natural question now is whether one can get quantitative information on the stochastic fractional nonlinear Schr\"odinger equation in $L^2(\mathbb{R}^n)$. The second difficulty lies in how to solve the problem of low regularity of noisy sample paths and non-compact Sobolev embeddings on unbounded regions. The third difficulty lies in how to find the appropriate techniques for handling nonlinear terms $|u|^{2\sigma}u$, which is different from the existing results.

This paper is organized as follows.  In Section 2,  we introduce some notations and state our main results in this present paper.
In Section 3, we use the stopping time technique, and deterministic and stochastic fractional Strichartz inequalities to prove the existence of the global existence of the original equation \eqref{li}. The equation \eqref{li} provides an infinite dimensional dynamical system that possesses a weak pullback mean random attractors in $L^\rho\left(\Omega; L^2\left(\mathbb{R}^n\right)\right)$ is proved in Section 4.

\section{Statement of the main results}
\subsection{Assumptions and preliminaries}
We introduce some notations throughout this paper. The capital letter $C$ denotes a positive constant, whose value may change from one line to another. The notation $C_p$ emphasizes that the constant only depends on the parameter $p$, while $C(\cdot\cdot\cdot)$ is used for the case of
more than one parameter. For $p\geq 2$, the notation $L^{p}$ denotes the Lebesgue space of complex-valued functions.  The inner product in $L^{2}\left(\mathbb{R}^{n}\right)$ is endowed with
\begin{equation*}
(g_1, g_2)=\int_{\mathbb{R}^{n}} g_1(x) \bar{g}_2(x) dx, \ \forall g_1, g_2 \in L^{2}\left(\mathbb{R}^{n}\right).
\end{equation*}

Given two separable Hilbert spaces $H$ and $\widetilde{H}$, the notation $L_{HS}(H,\widetilde{H})$ denotes the space of Hilbert-Schmidt operators from $H$ into $\widetilde{H}$. Let $\Phi: H\rightarrow \widetilde{H} $ be a bounded linear operator. The operator $\Phi$ is called the Hilbert-Schmidt operator if there is an orthonormal basis $(e_{k})_{k\in \mathbb{N}}$ in $H$ such that
\begin{equation*}
\|\Phi\|^2_{L_{HS}(H,\widetilde{H})}=\operatorname{tr} \Phi^* \Phi=\sum_{k=0}^{\infty}\left|\Phi e_k\right|_{\widetilde{H}}^2< \infty.
\end{equation*}
When $H =\widetilde{H}=L^2(\mathbb{R}^n; \mathbb{R})$, $L_{HS}\left(L^2(\mathbb{R}^n; \mathbb{R}), L^2(\mathbb{R}^n; \mathbb{R})\right)$ is simply denoted by $L_{HS}$.

For convenience, we often write $L^{p}(\Omega, \mathcal{F}; X)$ as $L^{p}(\Omega, X)$.
Let $\mathcal{D}$ be a collection of some families of nonempty bounded subsets of $L^{p}(\Omega, X)$ parametrized by $\varrho\in \mathbb{R}$, that is,
\begin{equation}
\mathcal{D}=\left\{D=\left\{D(\varrho)\subseteq L^{p}(\Omega, X): D(\varrho)\neq \emptyset \quad \text{bounded}, \varrho \in \mathbb{R}\right\}: D ~~\text{satisfies some conditions} \right\}.
\end{equation}
In the following, we recall some basic definitions of the mean random dynamical system, which come from \cite[Definition 2.1, Definition 2.2, Definition 2.4]{BW19} respectively.
\begin{definition}
A family $\Phi=\{\Phi(t,\varrho): t \in \mathbb{R}^{+}, \varrho \in \mathbb{R}\}$ of mapping from $L^{p}(\Omega, X)$ to $L^{p}(\Omega, X)$ is called a mean random dynamical system on $L^{p}(\Omega, X)$ if for all $\tau \in \mathbb{R}$ and $t, s \in \mathbb{R}^{+}$, the following conditions are fulfilled:\\
(i)~~$\Phi(0,\varrho)$ is the identity operator on $L^{p}(\Omega, X)$; \\
(ii)~~$\Phi(t+s,\varrho)=\Phi(t,\varrho+s)\circ \Phi(s,\varrho)$.

\end{definition}

\begin{remark}
We call such $\Phi$ a  mean random dynamical system simply because $\Omega$ is a probability space.
\end{remark}

\begin{definition}
A family $K=\{K(\varrho): \varrho \in \mathbb{R}\} \in \mathcal{D}$ is called a $\mathcal{D}$-pullback absorbing set for $\Phi$ if for every $\varrho \in \mathbb{R}$ and $D \in \mathcal {D}$, there exists $T=T(\varrho, D)>0$ such that for all $t\geq T$
\begin{equation}
\Phi(t, \varrho-t)\left(D(\varrho-t)\right)\subseteq K(\varrho).
\end{equation}
\end{definition}

If, in addition, $K(\varrho)$ is a closed nonempty subset of $L^{p}(\Omega, X)$ for every $\varrho \in \mathbb{R}$, then $K=\{K(\varrho): \varrho \in \mathbb{R}\}$ is called a closed $\mathcal{D}$-pullback absorbing set for $\Phi$. If $K(\varrho)$ is a weakly compact nonempty subset of $L^{p}(\Omega, X)$ for every $\varrho \in \mathbb{R}$, then $K=\{K(\varrho): \varrho \in \mathbb{R}\}$ is called a weakly compact $\mathcal{D}$-pullback absorbing set for $\Phi$.

%\begin{definition}
%A family $K=\{K(\tau): \tau \in \mathbb{R}\} \in \mathcal{D}$  is called a $\mathcal{D}$-pullback weakly attracting set of $\Phi$ in $L^{p}(\Omega, X)$ if for every
%$\tau \in \mathbb{R}, D \in \mathcal{D}$ and every weak neighborhood $\mathcal{N}^{w}(K(\tau))$ of $K(\tau)$, there exists $T=T(\tau, D, \mathcal{N}^{w}(K(\tau)))>0$
%such that for all $t\geq T$,
%\begin{equation}
%\Phi(t, \tau-t)\left(D(\tau-t)\right)\subseteq \mathcal{N}^{w}(K(\tau)).
%\end{equation}
%\end{definition}
%If, in addition, $K(\tau)$ is a weakly compact subset of $L^{p}(\Omega, X)$ for every $\tau \in \mathbb{R}$, then $K=\{K(\tau: \tau \in \mathbb{R})\}$ is called a
%$\mathcal{D}$-pullback weakly compact weakly attracting set for $\Phi$ in $L^{p}(\Omega, X)$.

\begin{definition}
A family $\mathcal{A}=\{\mathcal{A}(\varrho): \varrho \in \mathbb{R}\}\in \mathcal{D}$ is called a weak $\mathcal{D}$-pullback mean random attractor for $\Phi$ in $L^{p}(\Omega, X)$ if the following conditions are fulfilled:\\
(i) $\mathcal{A}(\varrho)$ is a weakly compact subset of $L^{p}(\Omega, X)$ for every $ \varrho\in \mathbb{R}$; \\
(ii) $\mathcal{A}$ is a $\mathcal{D}$-pullback weakly attracting set of $\Phi$  in $L^p(\Omega, X)$ in the sense of Definition; \\
(iii) $\mathcal{A}$ is the minimal element of $\mathcal{D}$ with property (i) and (ii); that is, if $B=\{B(\varrho): \varrho \in \mathbb{R}\} \in \mathcal{D}$ is a $\mathcal{D}$-pullback weakly compact weakly attracting set of $\Phi$ in $L^p(\Omega, X)$, then $\mathcal{A}(\varrho)\subseteq {B}(\varrho)$ for all $\varrho \in \mathbb{R}$.

\end{definition}
Here, we give some regularity assumption for the coefficient $f$ of system \eqref{li}. \\
\noindent ($\bf{A_{f}}$): For the function $f$ in \eqref{li}, we assume that $f: \mathbb{R}\times \mathbb{R}^n\times \mathbb{C}\rightarrow \mathbb{C}$ is continuous
and satisfies, for all $t, u\in \mathbb{R}$ and $x,y \in \mathbb{R}^{n}$,
\begin{align}\label{assum1}
 \mathbf{Im}\left(f(t,x,u)\bar{u}\right)&\leq \beta|u|^2+\psi_1(t,x),\\\label{assum2}
|f(t,x,u)|&\leq \psi_2(x)|u|+\psi_3(t,x),\\\label{assum3}
\frac{\partial f}{\partial u}(t,x,u)&\leq \psi_4(x), 
\end{align}
where $\beta>0$ is a constant, and
\begin{equation*}
\begin{aligned}
  &\psi_1 \in L^1(\mathbb{R},L^1(\mathbb{R}^n)), \quad\psi_2\in L^{\frac{nr}{4\alpha}}(\mathbb{R}^n),\\
  &\psi_3\in L_{loc}^r\left(\mathbb{R}, L^{p^{\prime}}(\mathbb{R}^n)\right), \psi_4\in L^{\frac{nr}{4\alpha}}(\mathbb{R}^n),
\end{aligned}
\end{equation*}
where $r=\frac{4(\sigma+1)\alpha}{n\sigma}$ and $p=2\sigma+2$.

\subsection{Main results}
This paper aims to study the weak pullback attractor for the damped stochastic fractional Schr\"odinger equation on $\mathbb{R}^n$ with $n\geq 2$. Our work is divided into two parts. In the first part, we will examine the well-posedness of a global solution for the systems with the fractional Laplacian operator, nonlinear term $|u|^{2\sigma}u$ and white noise. In the second part, we will prove that this equation provides an infinite dimensional dynamical system that possesses a weak pullback mean random attractors in $L^\rho\left(\Omega; L^2\left(\mathbb{R}^n\right)\right) $.

The first main result of this paper is about the existence and uniqueness of a global solution for the systems \eqref{li}.
\begin{theorem}
\label{theorem0}
 Let $n \geq 2$, $\alpha \in\left[\frac{n}{2 n-1}, 1\right)$, $0\leq\sigma<\frac{2 \alpha}{n-2 \alpha}$ and $\Phi\in L_{HS}\left(L^2(\mathbb{R}^n), L^2(\mathbb{R}^n)\right)$. Assume that $(r,p)=(\frac{4(\sigma+1)\alpha}{n\sigma}, 2\sigma+2)$, and that the radial initial data $u_0\in L^{\rho}\left(\Omega, \mathcal{F}_0;L^{2}(\mathbb{R}^{n})\right)$ for $\rho\geq r$. Then under Assumption ($\bf{A_{f}}$), for every $T>0$, there exists a unique global mild solution $u=u(t), t\in[0, T]$ of \eqref{li} such that
 $$
 u\in L^{\rho}(\Omega; L^{\infty}(0, T; L^{2}(\mathbb{R}^{n})) \cap L^{1}(\Omega; L^{r}(0, T; L^{p}(\mathbb{R}^{n}))),
 $$
 and $u(\cdot,\omega)\in \mathcal{C}(0, T; L^2\left(\mathbb{R}^n)\right)$ for $\mathbb{P}$-a.a. $\omega\in\Omega$.
\end{theorem}
\begin{remark}
To compare our result with previous work in the literature (see e.g., \cite{RW21}), we rigorously provide the range of values for the indicator parameter $\sigma$ to ensure the well-posedness of the global mild solution.  
\end{remark}
The following result gives a sufficient criterion for the existence of weak $\mathcal{D}$-pullback mean random attractors in $L^\rho\left(\Omega; L^2\left(\mathbb{R}^n\right)\right)$. We also require the following assumption about the function $\psi_1$:
\begin{equation}\label{assum4}
    \int^{\tau}_{-\infty}e^{(\gamma-\beta)\varrho s/2}
\|\psi_1\|^{\varrho/2}_{L^1_{x}}ds<\infty, \ \forall \varrho\in\mathbb{R}.
\end{equation}
\begin{theorem}
\label{theorem2}
Suppose that the conditions of Theorem \ref{theorem0} and \eqref{assum4} hold. If $\gamma>\beta$, then the stochastic equation \eqref{li} posses a unique weak $\mathcal{D}$-pullback random attractors $\mathcal{A}=\left\{\mathcal{A}: \varrho \in \mathbb{R}\right\}\in \mathcal{D}$ in $L^\rho\left(\Omega; L^2\left(\mathbb{R}^n\right)\right)$; that is,
\begin{description}
\item[(i)] $\mathcal{A}(\varrho)$ is weakly compact in $L^\rho\left(\Omega; L^2\left(\mathbb{R}^n\right)\right)$ for all $\varrho \in \mathbb{R}$.
\item[(ii)] $\mathcal{A}$ is a $\mathcal{D}$-pullback weakly attracting set of $\Phi$.
\item[(iii)] $\mathcal{A}$ is the minimal element of $\mathcal{D}$ with properties $(i)$ and $(ii)$.
\end{description}
\end{theorem}

\section{Existence of global solution}
\subsection{Radial local theory}
In this subsection, we use the radial Strichartz estimates and contraction mapping theorem in a suitable space to prove the local well-posedness of the equation \eqref{li}. Then based on the uniform estimates of the mass, the existence and uniqueness of a global solution is established. Here we only provide a brief proof, and more details can be found in \cite{BD99}. First, we recall radial Strichartz estimates for the fractional Schr\"odinger equation. Let the unitary group $S(t):=e^{-it(-\Delta)^{\alpha}}$.

\begin{lemma}
\label{0estimate}
For $n \geq 2$ and $\alpha \in \left[\frac{n}{2n-1}, 1\right)$, there exists a positive constant $C$ such that the following estimates hold:
  \begin{equation*}
  \begin{aligned}
   \left\|S(t)g\right\|_{L^{p}(\mathbb{R}, L^{q}{(\mathbb{R}^n)})}&\leq C \|g\|_{L^2{(\mathbb{R}^n})},\\
  \left\|\int_{0}^{t} S(t-\tau) h(\tau) d \tau\right\|_{L^{p}\left(\mathbb{R}, L^{q}{(\mathbb{R}^n)}\right)} & \leq C \|h\|_{L^{a^{\prime}}(\mathbb{R}, L^{b^{\prime}}{(\mathbb{R}^n)})},
  \end{aligned}
  \end{equation*}
  where $g$ and $h$ are radially symmetric and $(p, q),(a^{\prime}, b^{\prime})$ satisfy the fractional admissible condition,
\begin{equation*}
  p \in[2, \infty], \quad q \in[2, \infty), \quad(p, q) \neq\left(2, \frac{4 n-2}{2 n-3}\right), \quad \frac{2\alpha}{p}+\frac{n}{q}=\frac{n}{2}.
\end{equation*}
\end{lemma}

Next, we give the local well-posedness of the stochastic fractional nonlinear Schr\"odinger equation \eqref{li} with radially symmetric initial data in the space $L^{2}\left(\mathbb{R}^{n}\right)$.
\begin{proposition}(Radial local theory).
\label{theorem1}
 Let $n \geq 2$, $\alpha \in\left[\frac{n}{2 n-1}, 1\right)$, $0\leq\sigma<\frac{2 \alpha}{n-2 \alpha}$ and $\Phi\in L_{HS}$. Assume that $(r,p)=(\frac{4(\sigma+1)\alpha}{n\sigma}, 2\sigma+2)$, and that the radial initial data $u_0\in L^{\rho}\left(\Omega, \mathcal{F}_0;L^{2}(\mathbb{R}^{n})\right)$ for $\rho\geq r$. Then under Assumption $(\bf{A_{f}})$, there exist a stopping time $\tau^*\left(u_0, \omega\right)$ and a unique solution $u\in L^{\rho}(\Omega; L^{\infty}(0, \tau; L^{2}(\mathbb{R}^{n})) \cap L^{1}(\Omega; L^{r}(0, \tau; L^{p}(\mathbb{R}^{n})))$ to equation \eqref{li} for any $\tau<\tau^*\left(u_0, \omega\right)$. Moreover, we have almost surely,
\begin{equation*}
\tau^*\left(u_0, \omega\right)=+\infty \quad \text { or } \quad \limsup _{t\to\tau^*\left(u_0, \omega\right)}\|u(t)\|_{L^{2}}=+\infty.
\end{equation*}
\end{proposition}

\begin{proof}
We follow the approach of \cite[Theorem 2.1]{BD99}. For the damping term $i\gamma u$, we can define a new free operator $T_{\gamma,\alpha}\phi:=e^{-\gamma t}\mathcal{F}^{-1}(e^{-it|y|^{2\alpha}})*\phi$. It has the same Strichartz estimates as the unitary group $S(t)$ (see \cite{Sa15}, Proposition 2.7 and 2.9). So we can ignore the impact of the damping term and still use the unitary group $S(t)$ to represent the mild form of the equation \eqref{li} for simplicity.

Let $\theta \in \mathcal{C}_{0}^{\infty}(\mathbb{R})$ with $ \supp \theta \subset(-2,2), \theta(x)=1$ for $x \in[-1,1]$ and $0 \leq \theta(x) \leq 1$ for
$x \in \mathbb{R}$. Let $R>0$ and $\theta_{R}(x)=\theta\left(\frac{x}{R}\right)$. Let us fix $R\geq 1$.
We recall the truncated equation of equation \eqref{li}, that is
\begin{equation*}
\label{ut}
\begin{aligned}
u(t)=& S(t) u_{0}+i \int_{0}^{t} S(t-s)\left(\theta_{R}\left(|u|_{L^{r}\left(0, s ; L_{x}^{p}\right)}\right) |u(s)|^{2\sigma}u(s)\right) ds
-i \int_{0}^{t} S(t-s) d W(s)\\
&-\gamma \int_{0}^{t} S(t-s) u(s) ds-i\int_{0}^{t} S(t-s)\left(\theta_{R}\left(|u|_{L^{r}\left(0, s ; L_{x}^{p}\right)}\right) f\right)ds.
\end{aligned}
\end{equation*}
By using a fixed point iteration argument in a ball of $X_{T_0}$ for the norm in $L^{\infty}(0, T_{0}; L^{2}(\mathbb{R}^{n}))\cap L^{r}(0, T_{0}; L^{p}(\mathbb{R}^{n}))$, we can prove this theorem. Here we only need to prove that the different terms $i \int_{0}^{t} S(t-s) d W(s)$ and $i\int_{0}^{t} S(t-s)\left(\theta_{R}\left(|u|_{L^{r}\left(0, T_0 ; L_{x}^{p}\right)}\right) f\right)ds$ belong to the space $L^{\infty}\left(0, T_{0}; L^{2}\left(\mathbb{R}^{n}\right)\right)\cap L^{r}\left(0, T_{0}; L^{p}\left(\mathbb{R}^{n}\right)\right)$.

For stochastic integral $z(t):=i \int_{0}^{t} S(t-s) d W(s)$, it is standard that under our assumptions it is well-defined and has paths in $C\left([0,\infty);L^2(\mathbb{R}^n)\right)$(see \cite{DG92}, Theorem 6.10). And using the same method in \cite [Theorem 3.1]{BD99}, we also prove that it belongs to the space $L^{r}\left(0, T_{0}; L^{p}\left(\mathbb{R}^{n}\right)\right)$. In fact, since $z$ is a Gaussian process and $r>2$, by H\"older inequality, and  moment property for Gaussian process, we have
\begin{equation}
\begin{aligned}
\mathbb{E}\left(\int^{T_0}_{0}|z(s)|^{r}_{L^{p}(\mathbb{R}^n)}ds\right)&=\int^{T_0}_{0}\mathbb{E}\left(|z(s)|^{r}_{L^{p}(\mathbb{R}^n)}\right)ds
\leq C_1 \int^{T_0}_{0}\left(\mathbb{E}|z(s)|^{p}_{L^{p}(\mathbb{R}^n)}\right)^{\frac{r}{p}}ds \\
& \leq C_2 \int^{T_0}_{0}\left(\int_{\mathbb{R}^n}\left[\mathbb{E}(|z(s)|^2)\right]^{\frac{p}{2}}dx\right)^{\frac{r}{p}}ds,\\
\end{aligned}
\end{equation}
where $C_1$ and $C_2$ depend on $p$ and $r$.\\
By the results in \cite [Theorem 3.1]{BD99} and the semigroup $S(t)$ maps $L^2(\mathbb{R}^n)$ into $L^{r}(0,T_0; L^p(\mathbb{R}^n))$ for any $T_0\geq 0$, we have
\begin{equation}
\mathbb{E}\left(\int^{T_0}_{0}|z(s)|^{r}_{L^{p}(\mathbb{R}^n)}ds\right)\leq C_3 \sum_{i\in \mathbb{N}}|\Phi e_i|^2_{L^2(\mathbb{R}^n)}<\infty,
\end{equation}
where $C_3$ depends on $T_0,r$ and $p$.

For the term $i\int_{0}^{t} S(t-s)\left(\theta_{R}\left(|u|_{L^{r}\left(0, T_0 ; L_{x}^{p}\right)}\right) f\right)ds$, using the Strichartz estimates, assumption \eqref{assum2} and H\"older inequality, one has
\begin{equation}
\begin{aligned}
  &\left\|i\int_{0}^{t} S(t-s)\left(\theta_{R}\left(|u|_{L^{r}\left(0, T_0 ; L_{x}^{p}\right)}\right) f\right)ds\right\|_{L^r\left(0, T_0 ; L^p(\mathbb{R}^n)\right)}\\
  &\leq C\|f\|_{L^{r^{\prime}}\left(0, T_0 ; L^{p^{\prime}}(\mathbb{R}^n)\right)}\\
  &\leq C\|\psi_2(x)|u|\|_{L^{r^{\prime}}\left(0, T_0 ; L^{p^{\prime}}(\mathbb{R}^n)\right)}+C\|\psi_3\|_{L^{r^{\prime}}\left(0, T_0 ; L^{p^{\prime}}(\mathbb{R}^n)\right)} \\
  &\leq CT_0^{1-2/r}\left(\|\psi_2\|_{L_x^{\frac{nr}{4\alpha}}}\|u\|_{L^r\left(0, T_0 ; L^p(\mathbb{R}^n)\right)}+\|\psi_3\|_{L^r\left(0, T_0 ; L^{p^{\prime}}(\mathbb{R}^n)\right)}\right).
\end{aligned}
\end{equation}
This indicates that the term $i\int_{0}^{t} S(t-s)\left(\theta_{R}\left(|u|_{L^{r}\left(0, T_0 ; L_{x}^{p}\right)}\right) f\right)ds$ belongs to the space $L^{r}\left(0, T_{0}; L^{p}\left(\mathbb{R}^{n}\right)\right)$. Similarly, we can also prove that it belongs to the space $L^{\infty}\left(0, T_0; L^2(\mathbb{R}^n)\right)$.

Next, by the Strichartz estimates, assumptions \eqref{assum2}-\eqref{assum3} and H\"older inequality, we get
\begin{equation}
\begin{aligned}
&\left|\int_{0}^{t} S(t-s)(\theta_{R}\left(|u_1|_{L^{r}\left(0, T_0 ; L_{x}^{p}\right)}\right)f(x,u_1)-\theta_{R}\left(|u_2|_{L^{r}\left(0, T_0 ; L_{x}^{p}\right)}\right)f(x,u_2))ds\right|_{L^r\left(0, T_0; L^p(\mathbb{R}^n)\right)}\\
&\leq |\theta_{R}\left(|u_1|_{L^{r}\left(0, T_0 ; L_{x}^{p}\right)}\right)f(x,u_1)-\theta_{R}\left(|u_2|_{L^{r}\left(0, T_0 ; L_{x}^{p}\right)}\right)f(x,u_2)|_{L^{r^{\prime}}\left(0, T_0; L^{p^{\prime}}(\mathbb{R}^n)\right)}\\
&\leq |\left(\theta_{R}\left(|u_1|_{L^{r}\left(0, T_0 ; L_{x}^{p}\right)}\right)-\theta_{R}\left(|u_2|_{L^{r}\left(0, T_0 ; L_{x}^{p}\right)}\right)\right)f(x,u_2)|_{L^{r^{\prime}}\left(0, T_0; L^{p^{\prime}}(\mathbb{R}^n)\right)}\\
&\quad +|\theta_{R}\left(|u_1|_{L^{r}\left(0, T_0 ; L_{x}^{p}\right)}\right)\left(f(x,u_1)-f(x,u_2)\right)|_{L^{r^{\prime}}\left(0, T_0; L^{p^{\prime}}(\mathbb{R}^n)\right)} \\
&\leq C\|u_1-u_2\|_{L^{r}\left(0, T_0 ; L_{x}^{p}\right)}\|f(x,u_2)\|_{L^{r^{\prime}}\left(0, T_0; L^{p^{\prime}}(\mathbb{R}^n)\right)}\\
&\quad +C\|f(x,u_1)-f(x,u_2)\|_{L^{r^{\prime}}\left(0, T_0; L^{p^{\prime}}(\mathbb{R}^n)\right)}\\
&\leq CT_0^{1-2/r}\left(\|\psi_2\|_{L_x^{\frac{nr}{4\alpha}}}\|u\|_{L^r\left(0, T_0 ; L^p(\mathbb{R}^n)\right)}+\|\psi_3\|_{L^r\left(0, T_0 ; L^{p^{\prime}}(\mathbb{R}^n)\right)}\right)\|u_1-u_2\|_{L^{r}\left(0, T_0 ; L_{x}^{p}\right)}\\
&\quad +C\|\psi_4(x)(u_1-u_2)\|_{L^{r^{\prime}}\left(0, T_0; L^{p^{\prime}}(\mathbb{R}^n)\right)}\\
&\leq CT_0^{1-2/r}\left(\|\psi_2\|_{L_x^{\frac{nr}{4\alpha}}}\|u\|_{L^r\left(0, T_0 ; L^p(\mathbb{R}^n)\right)}+\|\psi_3\|_{L^r\left(0, T_0 ; L^{p^{\prime}}(\mathbb{R}^n)\right)}\right)\|u_1-u_2\|_{L^{r}\left(0, T_0 ; L_{x}^{p}\right)}\\
&\quad+ CT_0^{1-2/r}\|\psi_4\|_{L_x^{\frac{nr}{4\alpha}}}|u_1-u_2|_{L^r\left(0, T_0; L^p(\mathbb{R}^n)\right)}.
\end{aligned}
\end{equation}
The estimate in $L^{\infty}\left(0, T_0; L^2(\mathbb{R}^n)\right)$ is similar.
\end{proof}
Now, we start to give the following result about the behavior of  the mass, which is defined by
$M(u)=\|u\|^2$.

\begin{proposition}
\label{pro1}
Let $u_0$, $\sigma$, $f$ and $\Phi$ be as in Theorem \ref{theorem1}. Moreover, given $T>0$, let $\tau$ be a stopping time with $0<\tau<\min(\tau^{*}, T)$ almost surely. Then, for every $m\geq 1$, we have
\begin{equation}
  \label{mass}
  \begin{aligned}
\|u(t)\|^{2m}=&\|u_0\|^{2m}-2\gamma m\int_{0}^{\tau}\|u(s)\|^{2m}ds-2m\int^{\tau}_{0}\|u(s)\|^{2(m-1)}\operatorname{Im}\left(u(s),f(s)\right) ds\\
  &-2m\int^{\tau}_{0}\|u(s)\|^{2(m-1)}\operatorname{Im}\left(u(s),dW(s)\right)\\
  &+m\|\Phi\|^2_{L_{HS}}\int^{\tau}_{0}\|u(s)\|^{2(m-1)}ds\\
  &+2(m-1)m\int^{\tau}_{0}\|u(s)\|^{2(m-2)}\sum_{j=1}^{\infty}[\operatorname{Im}\left(u(s),\Phi e_j\right)]^2ds.
  \end{aligned}
  \end{equation}
\end{proposition}
\begin{proof}
A formal application of the It\^o formula to
$\|u(t)\|^{2m}$ yields \eqref{mass}. One can justify the computation by inserting several truncations and the local well-posedness argument. See \cite{BD99} for details.
\end{proof}

\subsection {Proof of Theorem \ref{theorem0}}
The proof of this theorem is a minor adaption of the proof of Theorem 2.1 in \cite{BD99}.
We only need to get the uniform boundedness of $\|u\|_{L^2}$ to ensure the global existence of the solution. Letting $m=1$ in the equality \eqref{mass}  and using the assumption \eqref{assum1}, we obtain that for every $T>0$, every stopping time $0<\tau<\min(\tau^{*}, T)$ and every time $t\leq \tau$,
\begin{equation}
\label{real}
\begin{aligned}
\left\|u(t)\right\|^2=&\left\|u_{0}\right\|^2-2\gamma \int_0^t\|u(s)\|^2ds
-2\int_{0}^{t}\rm{Im}\left(u(s), f(s)\right) d s \\
&-2\int^{t}_{0}\operatorname{Im}\left(u(s),dW(s)\right)+\|\Phi\|^2_{L_{HS}}t \\
\leq&\left\|u_{0}\right\|^2-2(\gamma-\beta)\int_0^t\|u(s)\|^2ds+2\|\psi_1\|_{L^1(0,T;{L_{x}^{1}})}\\
&-2\int^{t}_{0}\operatorname{Im}\left(u(s),dW(s)\right)+\|\Phi\|^2_{L_{HS}}t.\\
\end{aligned}
\end{equation}
Taking the supremum and using Burkholder-Davis-Gundy inequality and H\"older inequality yields
\begin{equation}
\begin{aligned}
\mathbb{E}\left[\sup_{t\in[0,T]}\left\|u(t)\right\|^2\right]\leq&\mathbb{E}\left[\left\|u_{0}\right\|^2\right]+2(\gamma-\beta)\mathbb{E}\left[\int_0^T\|u(s)\|^2ds\right]+2\|\psi_1\|_{L^1(0,T;{L_{x}^{1}})}\\
&+C\mathbb{E}\left[\left|\int^{T}_{0}\sum_{i=1}^{\infty}(u(s),\Phi e_i)^2ds\right|^{1/2}\right]+\|\Phi\|^2_{L_{HS}}T\\
\leq & C(T,\|\Phi\|^2_{L_{HS}},\mathbb{E}\left\|u_{0}\right\|^2, \|\psi_1\|_{L^1(0,T;{L_{x}^{1}})})+2(\gamma-\beta)\mathbb{E}\left[\int_0^T\|u(s)\|^2ds\right]\\
&+C(\|\Phi\|^2_{L_{HS}})\mathbb{E}\left[\left|\int_0^T\|u(s)\|^2ds\right|^{1/2}\right].
\end{aligned}
\end{equation}
Using Young's inequality and Gr\"onwall inequality, we obtain
\begin{equation}
\mathbb{E}\left[\sup_{t\in[0,T]}\left\|u(t)\right\|^2\right]\leq C(T,\|\Phi\|^2_{L_{HS}},\mathbb{E}\left\|u_{0}\right\|^2, \|\psi_1\|_{L^1(0,T;{L_{x}^{1}})}).
\end{equation}
These a priori estimations combined with local well-posedness imply the global existence of the unique solution.

\section{Existence for global attractor}

This section is devoted to the existence and uniqueness of weak pullback mean random attractors of the system \eqref{li} in $L^{\rho}\left(\Omega; L^{2}(\mathbb{R}^{n})\right)$.

For every $t\in\mathbb{R}^{+}$ and $\varrho \in \mathbb{R}$, we define a mapping $\Phi(t, \varrho) $ by
\begin{equation}
\Phi(t, \varrho)u_0=u(t+\varrho,\varrho, u_0), ~~\forall u_0 \in L^{\rho}\left(\Omega, \mathcal{F}_{\varrho};L^{2}(\mathbb{R}^{n})\right),
\end{equation}
By the uniqueness of the solution, we find that $\Phi(t+s, \varrho)=\Phi(t, s+\varrho)\circ \Phi(s, \varrho)$ for all $t, s \in \mathbb{R}^{+}$ and $\varrho \in \mathbb{R}$. Then $\Phi$ is a mean random dynamical system for \eqref{li} on $L^{\rho}\left(\Omega; L^{2}(\mathbb{R}^{n})\right)$ in the sense of \cite[Definition 2.9]{BW19}
\subsection{Moment estimates}
Denote by $\Theta$ the collection of all the families  $\mathcal{D}=\{B(\varrho)\subseteq L^{\rho}(\Omega; L^{2}(\mathbb{R}^{n})): \varrho \in \mathbb{R}\}$ be a family of nonempty bounded sets such that
\begin{equation}
\lim_{\varrho \rightarrow -\infty}e^{(\gamma-\beta)\varrho}\left(\sup_{u\in B(\varrho)}\|u\|^2_{L^{\rho}(\Omega; L^{2}(\mathbb{R}^{n})}\right)=0.
\end{equation}
\begin{lemma}
\label{lem2}
Let $u_0$, $\sigma$, $f$ and $\Phi$ be as in Theorem \ref{theorem1}. Assume that $\gamma>\beta$. Then for every $m \geq 1$, $\varrho \in \mathbb{R}$ and $B=\{B(\varrho)\}_{\varrho\in \mathbb{R}}\in \Theta$, there exists $T=T(\varrho, B)>0$ such that for all $t\geq T$, the solution
$u$ to \eqref{li} satisfies
\begin{equation}
\begin{aligned}
\mathbb{E} \|u(\varrho,\varrho-t,u_0)\|^{2m} &\leq \mathbb{E}\left[\|u_0\|^{2m}\right]+ C_1(m)\|\Phi\|_{L_{HS}}^{2 m}(\gamma-\beta)^{-m}\\
&\quad +C_2(m)(\gamma-\beta)^{1-m}e^{-(\gamma-\beta)m\tau}\int^{\tau}_{-\infty}e^{(\gamma-\beta)ms}
\|\psi_1(s)\|^m_{L^1_{x}}ds.\\
  \end{aligned}
\end{equation}
\end{lemma}
\begin{proof}
 By applying the It\^o's formula to $\|u(t)\|^{2m}$, we obtain
\begin{equation}
\label{zh}
\begin{aligned}
\|u(t)\|^{2m}=&\|u_0\|^{2m}-2\gamma m\int_{0}^{t}\|u(s)\|^{2m}dt-2m\int^{t}_{0}\|u(s)\|^{2(m-1)}\operatorname{Im}\left(u(s),f(s)\right) ds\\
&-2m\int^{t}_{0}\|u(s)\|^{2(m-1)}\operatorname{Im}\left(u(s),dW(s)\right)
+m\|\Phi\|^2_{L_{HS}}\int^{t}_{0}\|u(s)\|^{2(m-1)}ds\\
&+2(m-1)m\int^{t}_{0}\|u(s)\|^{2(m-2)}\sum_{j=1}^{\infty}[\operatorname{Im}\left( u(s),\Phi e_j\right)]^2ds.
\end{aligned}
\end{equation}
We know that the stochastic integral $\int^{t}_{0}\|u(s)\|^{2(m-1)}\operatorname{Im}\left( u(s),dW(s)\right)$ in \eqref{zh} is a martingale, so taking the expectation on both sides of \eqref{zh}, we have
\begin{equation*}
\begin{aligned}
\frac{d}{dt}\mathbb{E}\left[\|u(t)\|^{2m}\right]=&-2\gamma m \mathbb{E}\left[\|u(t)\|^{2m}\right]
-2m\mathbb{E}\left[\|u(t)\|^{2(m-1)}\operatorname{Im}\left( u(t),f(t)\right)\right]
+m\|\Phi\|^2_{L_{HS}}\mathbb{E}\left[\|u(t)\|^{2(m-1)}\right]\\
&+2(m-1)m\mathbb{E}\left[\|u(t)\|^{2(m-2)}\sum_{j=1}^{\infty}\left[\operatorname{Im}\left(u(t),\Phi e_j\right)\right]^2\right].
\end{aligned}
\end{equation*}
Using Young's inequality,  we obtain
\begin{equation*}
\begin{aligned}
&m\|\Phi\|_{L_{HS}}^2 \mathbb{E}\left[\|u\|^{2(m-1)}\right]+2(m-1) m \mathbb{E}\left[\|u\|^{2(m-2)} \sum_{j=1}^{\infty}\left[\operatorname{Im}\left(u, \Phi e_j\right)\right]^2\right]
 \leq m(2m-1)\|\Phi\|_{L_{HS}}^2 \mathbb{E}\left[\|u\|^{2(m-1)}\right] \\
& \leq \varepsilon_1 m \mathbb{E}\left[\|u\|^{2m}\right]+\|\Phi\|_{L_{HS}}^{2 m}\left(\frac{m-1}{m}\right)^{m-1}(2 m-1)^m \varepsilon_1^{1-m},
\end{aligned}
\end{equation*}
for $\varepsilon_1>0$. \\

Using the condition \eqref{assum1} and Young's inequality again, we obtain
\begin{equation*}
\begin{aligned}
\frac{d}{dt}\mathbb{E}\left[ \|u(t)\|^{2m}\right] &\leq -2\gamma m\mathbb{E}\left[\|u(t)\|^{2m}\right]-2m\mathbb{E}\left[\|u(t)\|^{2(m-1)}\operatorname{Im}\left(u(t),f(t)\right)\right]\\
 &\quad +\varepsilon_1 m  \mathbb{E}\left[ \|u(t)\|^{2m} \right]+C_1(m)\varepsilon_1^{1-m}\|\Phi\|_{L_{HS}}^{2 m} \\
&\leq -2\gamma m\mathbb{E}\left[\|u(t)\|^{2m}\right]+2\beta m\mathbb{E}\left[\|u(t)\|^{2m}\right]+2m\|\psi_1(t)\|_{L^1_x}\mathbb{E}\left[\|u(t)\|^{2(m-1)}\right]\\
&\quad +\varepsilon_1 m \mathbb{E} \left[\|u(t)\|^{2m}\right]+C_1(m) \varepsilon_1^{1-m}\|\Phi\|_{L_{HS}}^{2 m} \\
&\leq-2(\gamma-\beta) m\mathbb{E}\left[\|u(t)\|^{2m}\right]+\varepsilon_2 m \mathbb{E}\left[\|u(t)\|^{2m}\right]+2^m\|\psi_1(t)\|^m_{L^1_x}2^m\varepsilon_2^{1-m}\left(\frac{m-1}{m}\right)^{m-1}\\
&+\varepsilon_1 m \mathbb{E} \left[\|u(t)\|^{2m}\right]+C_1(m) \varepsilon_1^{1-m}\|\Phi\|_{L_{HS}}^{2 m},\\
\end{aligned}
\end{equation*}
for $\varepsilon_2>0$. Let $\varepsilon_1=\varepsilon_2 =\frac{1}{2}(\gamma-\beta)$. Then we have
\begin{equation}
\begin{aligned}
  \frac{d}{dt}\mathbb{E} \left[\|u(t)\|\right]^{2m} \leq& -\left(\gamma-\beta\right) m\mathbb{E}\|u(t)\|^{2m}
  +(\gamma-\beta)^{1-m}\left[\|\Phi\|_{L_{HS}}^{2 m} C_1(m)+C_2(m)\|\psi_1(t)\|^m_{L^1_x}\right].
\end{aligned}
\end{equation}
By means of Gr\"onwall inequality on $(\varrho-t,\varrho)$ and $\gamma>\beta$, we get
\begin{equation}
\begin{aligned}
\mathbb{E}\left[\|u(\varrho,\varrho-t,u_0)\|^{2m}\right] & \leq e^{-(\gamma-\beta) m t} \mathbb{E}\left[\|u_0\|^{2m}\right]\\
&\quad+(\gamma-\beta)^{1-m}\int_{\varrho-t}^{\varrho} e^{-(\gamma-\beta) m(\varrho-s)} \left[\|\Phi\|_{L_{HS}}^{2 m} C_1(m)+C_2(m)\|\psi_1(s)\|^m_{L^1_{x}}\right] ds \\
& =  e^{-(\gamma-\beta) m t} \mathbb{E}\left[\|u_0\|^{2m}\right]\\
&\quad+(\gamma-\beta)^{1-m}e^{-(\gamma-\beta)m\varrho}\int^{\varrho}_{\varrho-t}e^{(\gamma-\beta)ms}
\left[\|\Phi\|_{L_{HS}}^{2 m} C_1(m)+C_2(m)\|\psi_1\|^m_{L^1_{x}}\right]ds\\
& \leq e^{-(\gamma-\beta) m t} \mathbb{E}\left[\|u_0\|^{2m}\right]+ C_1(m)\|\Phi\|_{L_{HS}}^{2 m}(\gamma-\beta)^{-m}\left[1-e^{-(\gamma-\beta)mt}\right]\\
&\quad +C_2(m)(\gamma-\beta)^{1-m}e^{-(\gamma-\beta)m\varrho}\int^{\varrho}_{-\infty}e^{(\gamma-\beta)ms}
\|\psi_1(s)\|^m_{L^1_{x}}ds\\
&\leq \mathbb{E}\left[\|u_0\|^{2m}\right]+ C_1(m)\|\Phi\|_{L_{HS}}^{2 m}(\gamma-\beta)^{-m}\\
&\quad +C_2(m)(\gamma-\beta)^{1-m}e^{-(\gamma-\beta)m\varrho}\int^{\varrho}_{-\infty}e^{(\gamma-\beta)ms}
\|\psi_1(s)\|^m_{L^1_{x}}ds,\\
\end{aligned}
\end{equation}
for any $t \geq T$.

\end{proof}
In the following, we will give a sufficient criterion for the existence of weak $\mathcal{D}$-pullback mean random attractors in $L^\rho\left(\Omega; L^2\left(\mathbb{R}^n\right)\right)$.
\subsection{Proof of Theorem \ref{theorem2}}

Proof of Theorem  \ref{theorem2}. Given $\varrho\in \mathbb{R}$, denote by
\begin{equation}
\label{K}
\mathcal{K}(\varrho)=\left\{u\in \mathcal{D}: \mathbb{E}\left[\|u\|^{\rho}\right]\leq R(\varrho)\right \},
\end{equation}
where
\begin{equation}
\begin{aligned}
R(\varrho)&:=\mathbb{E}\left[\|u_0\|^{\rho}\right]+C_1(\rho)\|\Phi\|_{L_{HS}}^{\rho}(\gamma-\beta)^{-\rho/2}\\
&+C_2(\rho)(\gamma-\beta)^{1-\rho/2}e^{-(\gamma-\beta)\rho\varrho/2}\int^{\varrho}_{-\infty}e^{(\gamma-\beta)\rho s/2}
\|\psi_1\|^{\rho/2}_{L^1_{x}}ds.
\end{aligned}
\end{equation}
This implies that $\mathcal{K}(\varrho)$ given by \eqref{K} is a bounded closed convex subset in $L^\rho\left(\Omega; L^2\left(\mathbb{R}^n\right)\right)$, and hence it is weakly compact in $L^\rho\left(\Omega; L^2\left(\mathbb{R}^n\right)\right)$. By Lemma \ref{lem2}, we find that for every $(\varrho, {D})\in \mathbb{R}\times \mathcal{D}$, there exists $T:=T(\varrho, \mathcal{D})>0$ such that
\begin{equation}
\label{rr}
\sup_{t\geq T}\sup_{u_0\in \mathcal{D}(\varrho-t)}\|u(\varrho, \varrho-t, \theta_{-\tau}\omega, u_0)\|^{\rho}\leq  R(\varrho).
\end{equation}
The inequality \eqref{rr} tells us that for all $t\geq T$,
\begin{equation}
\Phi(t, \varrho-t)\mathcal{D}(\varrho-t)=u(\varrho, \varrho-t, \mathcal{D}(\varrho-t))\subseteq \mathcal{K}(\varrho).
\end{equation}
Therefore $\mathcal{K}=\{\mathcal{K}(\varrho): \varrho \in \mathbb{R}\}$ is an absorbing. Note that $\gamma > \beta$
\begin{equation}
\lim_{\varrho\rightarrow -\infty }e^{(\gamma-\beta) m \varrho}R(\varrho)=0.
\end{equation}
Then $\mathcal{K}$ is a weakly compact $\mathcal{D}$-pullback absorbing set for $\Phi$ in  $L^\rho\left(\Omega; L^2\left(\mathbb{R}^n\right)\right)$. By using the abstract result\cite[Theorem 2.13]{BW19}, we obtain the existence and uniqueness of weak $\mathcal{D}$-pullback mean random attractor $\mathcal{A} \in \mathcal{D}$ of $\Phi$.

%\lim_{t\rightarrow +\infty}\frac{1}{t}\int^{t}_{0}g(u(s;x_i))ds=\int_{V}fd\mu_i,
%\end{equation}
%With a short notation we write $u_i(t)=u(t;x_i)$. Then consider the difference $\omega=u_1-u_2$
%\end{proof}

\section{Acknowledgments}
 The research of Y. Zhang is supported by the Natural Science Foundation of Henan Province of China (Grant No. 232300420110).

\end{document}